\documentclass{article}
\usepackage[letterpaper,margin=1.0in]{geometry}
\usepackage{amsthm}

\usepackage{amsmath,amsfonts,bm,amssymb}

\def\eqref#1{equation~\ref{#1}}

\def\1{\bm{1}}

\def\eps{{\epsilon}}

\def\vzero{{\bm{0}}}

\def\ve{{\bm{e}}}

\def\vr{{\bm{r}}}
\def\vs{{\bm{s}}}

\def\vu{{\bm{u}}}
\def\vv{{\bm{v}}}
\def\vw{{\bm{w}}}
\def\vx{{\bm{x}}}
\def\vy{{\bm{y}}}
\def\vz{{\bm{z}}}

\def\mB{{\bm{B}}}

\def\mD{{\bm{D}}}

\def\mH{{\bm{H}}}

\DeclareMathAlphabet{\mathsfit}{\encodingdefault}{\sfdefault}{m}{sl}
\SetMathAlphabet{\mathsfit}{bold}{\encodingdefault}{\sfdefault}{bx}{n}

\def\gB{{\mathcal{B}}}

\def\gF{{\mathcal{F}}}

\def\gO{{\mathcal{O}}}

\def\gS{{\mathcal{S}}}

\def\sO{{\mathbb{O}}}

\def\sR{{\mathbb{R}}}

\newcommand{\E}{\mathbb{E}}

\DeclareMathOperator*{\argmin}{arg\,min}

\newtheorem{thm}{Theorem}[section]
\newtheorem{dfn}{Definition}[section]

\newtheorem{lem}{Lemma}[section]
\newtheorem{asm}{Assumption}[section]

\newtheorem{cor}{Corollary}[section]

\newtheorem{prop}{Proposition}[section]
\newtheorem{cond}{Condition}[section]

\newcommand{\norm}[1]{\left\| #1 \right\|}

\usepackage{mdframed}
\usepackage{hyperref}
\usepackage{enumitem}
\usepackage[utf8]{inputenc} 
\usepackage[T1]{fontenc}    
\usepackage{hyperref}       
\usepackage{url}            
\usepackage{booktabs}       
\usepackage{amsfonts}       
\usepackage{nicefrac}       
\usepackage{microtype}      
\usepackage{graphicx}       
\usepackage[numbers,sort]{natbib}
\usepackage{multirow}
\usepackage{bbding}
\usepackage{caption}
\graphicspath{{media/}}     
\usepackage{color}
\usepackage{algpseudocode,algorithm}
\usepackage{nicefrac}
\usepackage{cancel}
\usepackage{graphicx} 
\usepackage{booktabs}
\usepackage{makecell}
\usepackage{multirow}
\usepackage{amssymb,amsthm}
\usepackage{tcolorbox}
\usepackage{threeparttable}
\usepackage[table]{xcolor}
\usepackage{pifont}
\usepackage{multirow, hhline}
\definecolor{mydarkgreen}{RGB}{39,130,67}

\definecolor{mydarkred}{RGB}{192,47,25}

\newcommand{\Prob}{\mathbb P}

\usepackage{thmtools}
\usepackage{thm-restate}

\title{Near-Optimal Second-Order
Convex Finite-Sum Optimization under Individual Smoothness}

\author{Xinliang Zhang \textsuperscript{* 1} \qquad Lesi Chen \textsuperscript{* 1}\\
Zonghan Li \textsuperscript{1}
\qquad Jingzhao Zhang \textsuperscript{1} \\
\vspace{-2mm} \\
\normalsize{\textsuperscript{1}IIIS, Tsinghua University}\\
\normalsize{ \texttt{\{xinliang23,chenlc23,lizongha24\}@mails.tsinghua.edu.cn,  }} \\
\normalsize{\texttt{
jingzhaoz@mail.tsinghua.edu.cn
}}}

\begin{document}
\maketitle
\begingroup
\begin{NoHyper}
\renewcommand\thefootnote{*}
\footnotetext{Equal contributions.}
\end{NoHyper}
\endgroup
\begin{abstract}
We study the stochastic second-order oracle complexity of solving smooth convex finite-sum optimization with $n$ components. Very recently, Tsedenov \& Kovalev (2026) showed that the minimax optimal complexity under averaged smoothness and convexity is $\tilde \Theta(n + n^{6/7} \eps^{-2/7})$ by establishing nearly matching upper and lower bounds. In this paper, we show that individual smoothness and convexity provably enable further acceleration, where the lower complexity bound is $\Omega(n+ n^{5/7} \eps^{-2/7})$ under this stronger assumption. Then, by
establishing a novel second-order self-bounding inequality, we design a new Taylor-corrected lazy Newton method, which can lead to a better complexity of $\tilde{\mathcal{O}}(n + n^{5/7} \eps^{-2/7})$ by applying the optimal Monteiro-Svaiter acceleration, matching our lower bound up to only logarithmic factors.
\end{abstract}

\section{Introduction}

This paper studies the complexity of solving the finite-sum optimization problem:
\begin{align} \label{prob:main}
    \min_{\vx \in \sR^d} f(\vx) \triangleq \frac{1}{n} \sum_{i=1}^n f_i(\vx),
\end{align}
where each individual function $f_i(\vx)$ is smooth and convex. The convex finite-sum optimization problem covers many important applications in machine learning, including empirical risk minimization such as linear regression and logistic regression \citep{theodoridis2006pattern}.
We measure the complexity of an algorithm by the number of stochastic oracle calls, \textit{i.e.}, the derivative information of each component function $f_i$, to find an $\epsilon$-solution $\hat \vx \in \sR^d$ that satisfies $f(\hat \vx) - \min_{\vx \in \sR^d} f(\vx) \le \eps$. 

For full-batch first-order methods, the standard deterministic gradient descent achieves the stochastic first-order oracle (SFO) complexity of $\gO(n \epsilon^{-1})$, and Nesterov's accelerated gradient descent achieves $\gO(n \eps^{-1/2})$, which is optimal in $\eps$ for first-order methods \citep{nesterov2018lectures}. For deterministic second-order methods, Nesterov's accelerated cubic Newton method (ACRN) \citep{nesterov2006cubic,nesterov2008accelerating} achieves the stochastic second-order oracle (SSO) complexity of $\gO(n \eps^{-1/3})$, which is subsequently improved to $\gO(n \eps^{-2/7} \log \eps^{-1})$ by Monteiro-Svaiter acceleration \citep{monteiro2013accelerated}, and to recently $\gO(n \eps^{-2/7})$ by bisection-free techniques \citep{carmon2022optimal,kovalev2022first}, achieving optimal dependencies in $\eps$ that matches the lower bound for second-order methods \citep{arjevani2019oracle}.

However, the $n$ overhead per iteration is expensive in large-scale optimization, which inspires the study of stochastic or finite-sum methods. Under additional $\sigma^2$-bounded variance assumptions on the noise, the stochastic gradient descent (SGD) achieves the SFO complexity of $\gO(\sigma^2 \eps^{-2})$, which is suboptimal in $\eps$. Optimal~$\eps$ dependencies can be achieved by the celebrated variance reduction technique~\citep{johnson2013accelerating,roux2012stochastic,defazio2014finito,shalev2014accelerated,allen2018katyusha,lan2019unified,song2020variance,carmon2022recapp}. Notably, \citet{allen2018katyusha} proposed the Katyusha method that achieves the SFO complexity of $\gO(n \log \eps^{-1} + n^{1/2} \eps^{-1/2})$ when each $f_i$ is convex, nearly  matching the $\Omega(n+ n^{1/2} \eps^{-1/2})$ lower bound \citep{woodworth2016tight}.

Although optimal stochastic first-order methods are well studied, little is known on how to design optimal second-order ones, until very recently \citet{tsedenov2026optimal} proposed a novel method that achieves the SSO complexity of $  \gO( n \log n + n^{6/7} \eps^{-2/7})$. The authors also showed that their method is optimal under averaged smoothness and convexity by establishing a nearly matching lower bound of $\Omega(n + n^{6/7} \eps^{-2/7})$ for linear-spanned algorithms. 

One gap still remains: the above lower bound does not rule out the possibility of leveraging the individual smoothness and convexity that are common in practice \citep{johnson2013accelerating,defazio2014saga,allen2018katyusha,nguyen2017sarah} for further acceleration. To further investigate the optimal complexity under individual smoothness and convexity, we first prove a lower bound of $\Omega(n + n^{5/7} \eps^{-2/7})$ under this setting, suggesting a gap compared to existing upper bounds. Then, we seek an optimal stochastic second-order method under individual smoothness and convexity. In the literature of first-order methods, the source of acceleration of individual smoothness and convexity over the averaged ones mainly comes from the following self-bounding inequality of each $f_i$:
\begin{align} \label{eq:cocorcive-first-order}
\| \nabla f_i(\vx) - \nabla f_i(\vy) \|^2 \le 2 \ell D_{f_i}(\vx,\vy), \quad \forall \vx,\vy \in \sR^d,  
\end{align}
where each $f_i$ is convex with $\ell$-Lipschitz continuous gradients and $D_{f_i}(\vx,\vy) \triangleq f_i(\vx) - f_i(\vy) - \langle \nabla f_i(\vy), \vx - \vy \rangle$ is the Bregman distance induced by $f_i$. See, \textit{e.g.}, \citet[Theorem 2.1.5]{nesterov2018lectures} in  the textbook. By leveraging this first-order self-bounding inequality to obtain a better variance bound for stochastic variance-reduced estimators \citep{johnson2013accelerating}, the Katyusha algorithm \citep{allen2018katyusha} improves the guarantee of KatyushaX \citep{allen2018katyushax} without individual smoothness and convexity from $\gO(n + n^{3/4} \eps^{-1/2})$ to $\gO(n \log \eps^{-1} + n^{1/2} \eps^{-1/2})$. 

To exploit the individual smoothness and convexity assumptions for second-order methods, 
our work establishes a second-order self-bounding inequality when each $f_i$ has $L$-Lipschitz continuous Hessians:
\begin{align} \label{eq:cocorcive-second-order}
    \| \nabla f_i(\vx) - \nabla f_i(\vy) - \nabla^2 f_i(\vy) (\vx- \vy) \|^2 \le 4 L \| \vx - \vy \| D_{f_i}(\vx,\vy), \quad \forall \vx, \vy \in \sR^d.
\end{align}
To the best of our knowledge, this new inequality is not recorded in any references and is of independent interest. It gives a tighter control of the Taylor remainder on the left-hand side, which allows the design of a stochastic Taylor remainder estimator with a smaller error than \citet{tsedenov2026optimal}.  

Next, we design a stochastic Newton method that reuses the full gradient and Hessian at snapshot points every $n$ iterations but still converges after the correction of the stochastic Taylor estimator. This makes the problem close to the setting of reusing Hessians, whose near-optimal oracle complexity is well-understood by recent results \citep{doikov2023second,chen2025secondorder,adil2025balancing,chayti2023unified,zhang2026near,wu2026gradient}. By combining existing techniques for designing near-optimal methods under the latter setting~\citep{zhang2026near}, including optimal Monteiro-Svaiter acceleration \citep{monteiro2013accelerated,carmon2022optimal,kovalev2022first} and lazy Newton updates \citep{doikov2023second,chen2026faster,zhang2026near}, we develop the Accelerated Taylor-corrected Lazy Newton (A-TLN) method that achieves the SSO complexity of 
$\tilde{\mathcal{O}}(n + n^{5/7}\epsilon^{-2/7})$, matching our newly established lower bound up to logarithmic factors and significantly improving upon existing methods without individual smoothness and convexity, as compared in Table~\ref{tab:res}.

\begin{table*}[htbp]
  \caption{Comparison of first-order and second-order stochastic methods for solving Problem~(\ref{prob:main}) under different smoothness assumptions. The averaged smooth lower bounds \citep{zhou2019lower,tsedenov2026optimal} also satisfy individual convexity.}
  \label{tab:res}
  \centering
        \begin{tabular}{c c c c}
        \hline
        Oracle & Smoothness \& Convexity  & Upper Bound & Lower Bound \\ 
        \hline \addlinespace
        \multirow{2}{*}{First-Order} & Averaged
        & \makecell[c]{$\gO(n+ n^{3/4} \eps^{-1/2} )$ \\
        (\citet{allen2018katyushax})} & \makecell[c]{$\Omega(n+ n^{3/4} \eps^{-1/2} )$ \\
        (\citet{zhou2019lower})}  \\
        \cline{2-4} \addlinespace
        & Individual &  \makecell[c]{ $\tilde \gO(n+n^{1/2} \eps^{-1/2})$ \\ (\citet{allen2018katyusha}) } & \makecell[c]{$\Omega(n+n^{1/2} \eps^{-1/2}) $ \\ (\citet{woodworth2016tight})}  \\
        \hline \addlinespace
        \multirow{2}{*}{Second-Order} 
        & Averaged & \makecell[c]{ $\tilde \gO(n+n^{6/7} \eps^{-2/7})$ \\ (\citet{tsedenov2026optimal}) } & \makecell[c]{$\Omega(n+n^{6/7} \eps^{-2/7})$ \\ (\citet{tsedenov2026optimal})} \\
        \cline{2-4} \addlinespace
       & Individual & \makecell[c]{ $\tilde \gO(n+n^{5/7} \eps^{-2/7})$ \\ (This Paper)} & \makecell[c]{$\Omega(n+n^{5/7} \eps^{-2/7})$ \\ (This Paper) } \\ 
        \hline
        \end{tabular}
\end{table*}


\paragraph{AI Usage.} 
The lower bound was proved by the authors, while the upper bound was discovered by ChatGPT 6 Astra. To ensure reproducibility, the corresponding chat link is publicly available \footnote{https://chatgpt.com/share/6acef320-0c3c-83ea-bafb-ac1554d1ea4e}.




\section{Related Works}


\subsection{Finite-Sum Optimization Under Individual Smoothness and Convexity}
When each $f_i(\vx)$ is convex and smooth, variance-reduced stochastic first-order methods such as SAG \citep{roux2012stochastic}, SAGA \citep{defazio2014saga}, Finito \citep{defazio2014finito}, MISO \citep{mairal2015incremental}
SDCA \citep{shalev2014accelerated}, SVRG \citep{johnson2013accelerating,xiao2014proximal}, and SARAH \citep{nguyen2017sarah} achieve the SFO complexity of 
$\gO( (n+ \kappa)\log \eps^{-1})$ for strongly convex problems with condition number $\kappa$, or the complexity of 
$\gO((n + \epsilon^{-1}) \log \epsilon^{-1})$ for convex problems using the black-box reduction in \cite{allen2016optimal}. By using these methods to solve the proximal subproblems in accelerated proximal-point methods, the Catalyst/APPA acceleration \citep{frostig2015regularizing,lin2018catalyst} can achieve the near-optimal complexity of $\gO( (n + n^{1/2} \eps^{-1/2}) \log \epsilon^{-1})$. Improved direct accelerated methods Katyusha \citep{allen2018katyusha}/Varag \citep{lan2019unified} can achieve a faster rate of $\gO(n \min \{ \log \eps^{-1} , \log n\} + n^{1/2} \eps^{-1/2} )$. More recently, VRADA \citep{song2020variance}/RECAPP \citep{carmon2022recapp} improved the ISO complexity to $\gO(n \log \log n + n^{1/2} \eps^{-1/2} )$, which nearly matches the lower bound of $\Omega(n + n^{1/2} \eps^{-1/2})$ lower bound~\citep{woodworth2016tight} up to a $\log \log n$ factor in the first term.


\subsection{Finite-Sum Optimization Under Averaged Smoothness and Convexity}

Motivated by applications such as PCA and SVD \citep{saad2011numerical} where individual convexity may not hold, there is also a line of work studying the setting under averaged smoothness and convexity. 
\citet{shalev2016sdca} proposed the dual-free SDCA  that achieves the SFO complexity of $\gO( (n + \epsilon^{-2}) \log \epsilon^{-1})$; \citet{allen2016improved} showed that SVRG with a shaper analysis achieves the complexity of $\gO( n \log \eps^{-1} + \epsilon^{-2})$. Finally, \citet{allen2018katyushax} proposed the KatyushaX method that achieves the optimal SFO complexity of
$\gO( n + n^{3/4} \epsilon^{-1/2})$ that exactly matches the lower bound \citep{zhou2019lower}. Very recently, \citet{tsedenov2026optimal} proposed a near-optimal variance-reduced second-order algorithm that achieves a near-optimal SSO complexity of $\gO(n\log n + n^{6/7} \eps^{-2/7})$, which matches their lower bound of $\Omega(n+ n^{6/7} \eps^{-2/7})$ under averaged smoothness and convexity up to a factor of $\log n$. But the optimal ISO complexity under individual smoothness and convexity remains open.

\subsection{Finite-Sum Optimization Without Convexity}

There is also a line of work studying nonconvex finite-sum optimization, focusing on finding a weaker solution notion of stationary points or local minima.
\citet{reddi2016stochastic,allen2016variance,lei2017non} extended SVRG to nonconvex optimization and achieved the SFO complexity of $\gO(n+ n^{2/3} \epsilon^{-2})$ for finding $\epsilon$-stationary points. Later, \citet{fang2018spider,zhou2020stochastic} proposed SPIDER/Nested-SVRG that achieves an optimal SFO complexity of $\gO(n+n^{1/2}\epsilon^{-2})$ under both averaged smoothness \citep{fang2018spider,zhou2019lower} and individual smoothness \citep{peng2026dense}. 
\citet{agarwal2017finding} proposed a faster method that requires $\tilde \gO(n \eps^{-3/2} + n^{3/4} \eps^{-7/4} )$ stochastic Hessian-vector-product oracles for finding an $\eps$-local minima under additional second-order smoothness.

The same technique also generalizes to stochastic second-order methods for nonconvex optimization. Under the individual smoothness, \citet{zhou2019stochastic} proposed the stochastic variance-reduced cubic Newton (SVRC) method that achieves the SSO complexity of $\tilde \gO(n+ n^{4/5} \epsilon^{-3/2} )$, \citet{shen2019stochastic,zhou2022dimension} used SPIDER-type recursive estimators to improve the complexity to $\gO(n+n^{3/4} \epsilon^{-3/2} )$, which is optimal due to the recent lower bounds \citep{emmenegger2022oracle,wu2026matching}. In addition, \citet{wu2026matching} showed that the same upper and lower bound of $\Theta(n+ n^{3/4} \eps^{-3/2})$ also holds under the averaged smoothness assumption. Differing from their finding in nonconvex optimization, our work demonstrates a clear separation in the optimal SSO complexity for convex finite-sum optimization under individual and averaged smoothness.

\subsection{Optimization with Lazy Second-Order Oracles}

Another line of work considers algorithms under a simpler setting that uses full gradients and stochastic Hessians, where the gradient queries are not counted into the complexity. \citet{wang2019stochastic,zhou2019stochastic} proposed the lite stochastic variance-reduced cubic (Lite-SVRC) method that achieve the stochastic Hessian oracle (SHO) complexity of $\gO(n +n^{2/3} \epsilon^{-3/2})$ for finding $\eps$-stationary points for nonconvex finite-sum optimization. \citet{doikov2023second} proposed the lazy Hessian technique that reuses Hessians every $n$ iterations, and showed an SHO complexity to $\gO(n + n^{1/2}\eps^{-3/2})$. \cite{adil2025balancing,chen2026faster} proposed improved methods that achieve the $\gO(n + n^{1/3} \epsilon^{-3/2})$ SHO complexity, which is also recently shown to be tight for the nonconvex case in \citet{wu2026gradient}. For the convex case that is more related to this paper, \citet{zhang2026near} recently proposed a method that achieves the SHO complexity of $\tilde \gO(n + n^{1/7} \eps^{-2/7})$ that nearly matches their lower bound. 



\section{Preliminaries}

\paragraph{Notations.} 

In this paper, we use $\| \,\cdot\,\|$ to denote the Euclidean norm of a vector or the spectral norm of a matrix. We let $v_{[j]}$ be the $j$-th coordinate of a vector $\vv$ and let $H_{[i,j]}$ be the $(i,j)$-th element of a matrix $\mH$. We also use standard asymptotic notation, which denotes $g = \gO(h)$ (or $h = \Omega(g)$) if $g(s) \le c h(s)$ for some constant $c > 0$ and all $s \in \gS$. We write $h = \Theta(g)$ if $h = \gO(g)$ and $h = \Omega(g)$. The notation $\tilde \gO(\,\cdot\,)$, $\tilde \Omega(\,\cdot\,)$, and $\tilde \Theta(\,\cdot\,)$ hide all logarithmic factors.

\subsection{Main Assumptions}

For Problem (\ref{prob:main}), we assume that each individual function $f_i$ is convex and has Lipschitz continuous Hessians, and their average has a bounded minimizer. 


\begin{asm} \label{asm:main}
We say $\{f_i\}_{i=1}^n \in \gF(R,L,n)$ if for some $d>0$, these functions satisfy the following.
\begin{enumerate}[label=(\alph*)]
    \item Each function $f_i: \sR^d \rightarrow \sR$ is convex and has $L$-Lipschitz continuous Hessians, \textit{i.e.},
    \begin{align*}
        \Vert \nabla^2 f_i(\vx) - \nabla^2 f_i(\vy) \Vert \le L \Vert \vx-  \vy \Vert, \quad \forall i \in [n], ~~ \forall \vx, \vy \in \sR^d.
    \end{align*}
    \item The average $f = (1/n) \sum_{i=1}^n f_i$ has 
    at least one minimizer $\vx^* \in \arg \min_{\vx \in \sR^d} f(\vx)$ such that $\Vert \vx^* \Vert \le R$. 
\end{enumerate}
\end{asm}
 
Very recently, \citet{tsedenov2026optimal} considered the complexity under a more general setting under averaged convexity and smoothness, which relaxes Assumption \ref{asm:main}(a) to (a1) the averaged function $f = (1/n) \sum_{i=1}^n f_i$ is convex and allows each individual function $f_i$ to be nonconvex (a2) the individual functions $\{f_i \}_{i=1}^n$ are second-order $\bar L$-averaged smooth, \textit{i.e.},  $ (1/n) \sum_{i=1}^n \| \nabla
^2 f_i(\vx) - \nabla^2 f_i(\vy) \|^2 \le \bar L^2 \| \vx - \vy \|^2$. The individual convexity and smoothness we consider are strictly stronger conditions. Nonetheless, they are standard and naturally satisfied in many machine learning applications such as logistic regression or soft $\ell_\infty$-regression \citep{doikov2023second,nesterov2005smooth,bullins2020highly}. It is clear that the smoothness constants satisfy the following relationship.

\begin{prop}[{\citet[Proposition 1]{luo2022complexity}}] \label{prop:rela-smoothness}
If each individual function $f_i$ has $L$-Lipschitz continuous Hessians, then $\{f_i \}_{i=1}^n$ are also second-order $L$-averaged smooth. Conversely, if the individual functions  $\{f_i \}_{i=1}^n$ are second-order $\bar L$-averaged smooth, then each  $f_i$ has $\sqrt{n}\bar{L}$-Lipschitz continuous Hessians. Equivalently,
\begin{align} \label{eq:rela-smoothness}
    \bar L \le L \le \sqrt{n} \bar L, \quad \text{when} ~~ L < + \infty.
\end{align}
\end{prop}
The main goal of this work is to show a better complexity 
by leveraging the convexity of each individual function
when assuming bounded $L$ rather than $\bar L$.
\subsection{Oracle Complexity}

Under Assumption \ref{asm:main}, we aim to find an $\epsilon$-solution of the averaged function, defined as follows.

\begin{dfn} \label{dfn:eps-solu}
We say $\hat \vx \in \sR^d$ is an $\epsilon$-solution to Problem (\ref{prob:main}) if we have $f(\hat \vx) - f^* \le \epsilon$, where $f^* = \min_{\vx \in \sR^d} f(\vx)$ is the optimal value of $f$.
\end{dfn}

In the well-studied setting of finite-sum first-order optimization \citep{agarwal2015lower,woodworth2016tight,zhou2019lower}, the oracle complexity is defined as the number of stochastic first-order oracle (SFO) calls required to find an $\eps$-solution, where SFO is the mapping defined as $(i, \vx) \mapsto \nabla f_i(\vx) $ that returns the gradient of an individual function given the index $i$ and query point $\vx$. In this paper, we consider the complexity with access to the stronger 
stochastic second-order oracle (SSO) \citep{emmenegger2022oracle,arjevani2017oracle} that additionally returns the Hessian of the selected individual function.

\begin{dfn} \label{dfn:sso}
The stochastic second-order oracle $\sO$ is defined as the mapping $(i,\vx) \mapsto (\nabla f_i(\vx), \nabla^2 f_i(\vx))$. 
\end{dfn}

This paper studies the complexity of the following zero-respecting periodic stochastic second-order algorithms, which assume that the algorithm periodically queries a full-batch second-order oracle once every~$n$ iterations, and generates the next query on the support of previous oracle calls.

\begin{dfn} \label{dfn:zero-respecting-alg}
Let the support of a vector $\vv$ and matrix $\mH$ be defined as $\operatorname{supp}(\vv) = \{j \mid v_{[j]} \neq 0\}$ and $\operatorname{supp}(\mH) = \{j \mid \exists k, H_{[j,k]} \neq 0 \text{ or } H_{[k,j]} \neq 0\}$, respectively.
A zero-respecting periodic stochastic second-order algorithm $\mathcal A$ operating on $f=(1/n) \sum_{i=1}^n f_i$ generates sequence of iterates $\{\vx_t\}_{t=1}^{T}$ satisfies
\[
\operatorname{supp}(\vx_t) \subseteq \bigcup_{s=1}^{t} \operatorname{supp}(\sO_s), \quad \forall t \ge 1; \qquad \vx_0 = \vzero,
\]
where the support of an oracle response $\sO_s$ is
\[
\operatorname{supp}(\sO_s) := \begin{cases}
\bigcup_{i=1}^n \left( \operatorname{supp}(\nabla f_i(\vx_{s-1})) \cup \operatorname{supp}(\nabla^2 f_i(\vx_{s-1})) \right), & \textit{if} \quad (s-1) \bmod n =0; \\
\operatorname{supp}(\nabla f_{I_s}(\vx_{s-1})) \cup \operatorname{supp}(\nabla^2 f_{I_s}(\vx_{s-1})), ~~ I_t\sim\operatorname{Unif}([n]) & \text{otherwise}.
\end{cases}
\]
\end{dfn}


The assumptions above are standard and contain a lot of practical algorithms. The zero-respecting assumption $\operatorname{supp}(\vx_t) \subseteq \bigcup_{s=1}^{t} \operatorname{supp}(\sO_s)$ is usually satisfied by iterative optimization algorithms \citep{carmon2020lower,nesterov2018lectures}. The periodic assumption (querying full-batch oracles if $(s-1) \bmod n =0$) is not only typically satisfied by variance-reduced methods \citep{johnson2013accelerating,fang2018spider,zhou2020stochastic,tsedenov2026optimal}, but also without loss of generality since the same lower bound up to constants holds for algorithms that query the full-batch oracles at any iteration by a simple reduction. See, \textit{e.g.}, \citet[Lemma 3.3]{zhang2026near}. Finally, our algorithm class is strictly larger than the linear-span algorithms with pure random indexes in \citet{tsedenov2026optimal}.

\section{Lower Complexity  Bound}


The main starting point of this paper is that the tight lower bound of $\Omega(n+ n^{6/7} (\bar L / \eps)^{2/7}  )$ under $\bar L$-averaged smoothness in \citet{tsedenov2026optimal} is weakened under individual smoothness in light of Proposition~\ref{prop:rela-smoothness}. 
Briefly speaking,  by observing that the individual smoothness of their constructed functions 
achieves equality with the right-hand side of \eqref{eq:rela-smoothness}, \textit{i.e.}, $\bar L = L / \sqrt{n}$, we can obtain the following lower bound after rescaling the functions to satisfy the $L$-individual smoothness:
\[
\Omega\left(n+ n^{6/7} \left( \frac{L}{\sqrt{n} \eps} \right)^{2/7}  \right  ) = \Omega\left(n+ n^{5/7} \left( \frac{L}{\eps} \right)^{2/7}  \right ).
\]

Formally, for $a=R/\sqrt{2K}$, $\lambda = L/48$, and $c(j)=1+((j-1)\bmod n)$, we construct each component of the hard instance $f_i: \sR^{2K} \rightarrow \sR$ as
\begin{equation}\label{eq:instance}
  f_i(\vx)=\lambda\sum_{\substack{1\le j\le 2K\\c(j)=i}}\psi_j(\vx),\qquad
  \psi_1(\vx)=|a-x_{[1]}|^3,\quad
  \psi_j(\vx)=|x_{[j-1]}-x_{[j]}|^3\quad(j\ge2),
\end{equation}


Under Construction (\ref{eq:instance}), the averaged function in is then the hard instance for full-batch second-order optimization \citep{arjevani2019oracle} scaled by $1/n$, that is, $f(\vx) \triangleq (1/n)\sum_{i=1}^n f_i(\vx) = (\lambda/ n) \sum_{j=1}^{2K} \psi_j(\vx)$, which is a zero-chain~\citep{carmon2020lower} such that an iterative algorithm can only explore the new coordinates one by one.  The assignment of $c(j)$ makes optimizing the finite-sum problem much harder for $n > 1$ since a random index $I_t$ can only reveal the next link with probability~$1/n$, which significantly slows down the progress. Using the terminology in \citet{arjevani2023lower}, this construction is essentially a probability $1/n$-zero-chain, thus, its last coordinates take $\Omega(nK)$ iterations to activate with high probability. 


\begin{lem}\label{lem:chain}
For $t\ge n$, let $K=\lceil 8T/n\rceil$. For any zero-respecting algorithm in Definition \ref{dfn:zero-respecting-alg}, the algorithm's output $\vx_T$ satisfies $(\vx_{T})_{[2K]} = 0$ for its last coordinate with probability at least $13/16$.
\end{lem}


For any vector $\vx \in \sR^{2K}$ whose last coordinate is zero, we can use Jensen's inequality to give a lower bound on the suboptimality gap of $f$ similar to the analysis in \citet{arjevani2019oracle}.

\begin{lem}\label{lem:gap}
For any vector $\vx \in \sR^{2K}$ that satisfies 
$\vx_{[2K]} = 0$, we have
\begin{equation}\label{eq:gap}
 f(\vx)-f^\star\ge \frac{\lambda a^3}{4 n K^2}.
\end{equation}
\end{lem}

This lemma immediately gives the following lower bound by substituting the parameter settings of $a = R/ \sqrt{2K}$ and $\lambda = L/48$ in our construction.

\begin{thm}\label{thm:lower}
For $n\ge1$, $L,R>0$, and $T\ge n$, set $K=\lceil8T/n\rceil$. For any algorithm in Definition \ref{dfn:zero-respecting-alg}, the hard instance $f: \sR^{2K} \rightarrow \sR$ in \eqref{eq:instance} satisfies Assumption \ref{asm:main} and the algorithm output $\vx_T$ satisfies
\[
f(\vx_T) - f^\ast \ge \frac{LR^3}{48 n (2K)^{7/2}}, \quad \text{with probability } 13/16.
\]
Consequently, the total iteration complexity to find an $\eps$-solution is at least $\Omega\bigl(n^{5/7}(LR^3/\epsilon)^{2/7}\bigr)$.
\end{thm}

Note that an additional $\Omega(n)$ trivially holds by randomly hiding a non-zero function in one of the $n$ components. See, \textit{e.g.}, \citet{woodworth2016tight}. Therefore, the above theorem implies the total lower bound of $\Omega(n + n^{5/7} \eps^{-2/7})$ for zero-respecting stochastic second-order algorithms. We believe that the same lower bound also holds for randomized algorithms by modifying our construction to a robust zero-chain \citep{arjevani2023lower,garg2021near}, and leave a rigorous proof for future work.

\section{Upper Complexity Bound}

Very recently, \citet{tsedenov2026optimal} established 
 the upper bound of $\tilde \gO(n + n^{6/7} \eps^{-2/7})$ 
under averaged smoothness and convexity, which definitely apply to our setting. However, their result still has a gap of $n^{1/7}$ compared to our lower bound in Theorem \ref{thm:lower} under the stronger individual smoothness and convexity. This section closes this gap by proposing a new stochastic second-order algorithm. 

\subsection{Second-Order Self-Bounding Inequality}
Before proving our new upper bound, we first review the key self-bounding inequality for first-order variance-reduced methods \citep{johnson2013accelerating,allen2018katyusha} and then establish an analogy for second-order methods. For the first-order stochastic variance-reduced gradient (SVRG) estimator
\citep{johnson2013accelerating,allen2018katyusha} \(\vu(\vv; \vw)=\nabla f_i(\vv)-\nabla f_i(\vw )+\nabla f(\vw)\) at query $\vv$ using a snapshot~$\vw$ and random sampled index $i\sim \text{Unif}([n])$, its second-order moment satisfies
\[
\E \| \vu(\vv;\vw) \|^2 \le \frac{2}{n} \sum_{i=1}^n ( \| \nabla f_i(\vv) - \nabla f_i(\vx^*) \|^2 + \| \nabla f_i(\vw) - \nabla f_i(\vx^*) \|^2) \le  4 \ell ( D_{f}(\vv,\vx^*) + D_f(\vw,\vx^*)),
\]
where the last step uses the self-bounding inequality (\ref{eq:cocorcive-first-order}) of $ \|\nabla f_i(\vx) - \nabla f_i(\vy) \|^2 \le 2 \ell D_{f_i}(\vx,\vy)$ under first-order $\ell$-individual smoothness and convexity, without which one can only obtain a looser bound of $ 2 \ell^2 (\| \vv - \vx^* \|^2 + \| \vw - \vx^* \|^2 ) $ for the right-hand side using the $\ell$-averaged smoothness. 

For designing variance-reduced second-order methods \citep{zhou2019stochastic,wang2019stochastic,shen2019stochastic,tsedenov2026optimal}, we typically require an upper bound of the stochastic Taylor residual estimator $\vr(\vv;\vw) = \nabla f_i(\vv) - \nabla f_i(\vw) - \nabla^2 f_i(\vw) (\vv- \vw)$ with  $i\sim \text{Unif}([n])$. \citet{tsedenov2026optimal} bounds $\E \| \vr(\vv;\vw) \|^2 \le (\bar L^2/4) \| \vv - \vw \|^4$ under second-order $\bar L$-averaged smoothness, which does not rely on individual convexity. In this paper, we exploit the following tighter bound under the additional smoothness and convexity assumption on each $f_i$:
\begin{lem} \label{lem:bregman-ieq}
If $f: \mathbb{R}^d \rightarrow \sR$ is convex and has $L$-Lipschitz continuous Hessians, then, for all $\vv,
\vw\in\mathbb{R}^d$,
\begin{align} \label{eq:self-bounding}
   \| \nabla f(\vv) - \nabla f(\vw) - \nabla^2 f(\vw) (\vv- \vw)  \|^2 \le 4 L \| \vv - \vw \| D_f(\vv,\vw). 
\end{align}
\end{lem}
This new second-order self-bounding inequality extends the textbook result of \citet[Theorem 2.1.5]{nesterov2018lectures} in (\ref{eq:cocorcive-first-order}) to second order.
It bounds the Taylor residual directly by the local Bregman divergence $D_f(\vv,\vw):=f(\vv)-f(\vw)-\langle\nabla f(\vw),\vv-\vw\rangle$ rather than global Hessian Lipschitzness, utilizing the convex geometry of each individual function.

\begin{algorithm*}[t]  
\caption{\texttt{A-TLN}$(f, \vx_0, \gamma, T, S,K) $} \label{alg:A-LEN}
\begin{algorithmic}[1] 
\Function{Cubic-Proximal}{$f,\vx, \gamma, S,K$} 
\State Set $H=L+2\gamma$, $\vz_0=\mathcal{C}_H(\vx)$ and $\vy_0=\mathcal{C}_H(\vz_0)$
\State Set $\Delta_0=\frac{2\sqrt2}{3\sqrt\gamma}
\|\nabla f_{\vx,\gamma}(\vy_0)\|^{3/2}$
\Repeat \label{line:repeat-begin}
\For{$s=0,\ldots,S-1$}
\State {\color{blue}$\vy_{s+1}=\texttt{TLN}(f_{\vx,\gamma}, \vy_s, \Delta_s, H , \gamma,K, 1 / (4S))$} \Comment{Implemented by Algorithm \ref{alg:VR-LEN-Gap}}
\State $\Delta_{s+1}=3\Delta_s/4$
\EndFor 
\State Set $\hat{\vy}=\mathcal{C}_H(\vy_S)$ 
\Until $\|\nabla f_{\vx,\gamma}(\hat{\vy})\|\leq (\gamma/32)\|\hat{\vy}-\vx\|^2$ \label{line:repeat-end} \\
\Return $\hat \vy$.
\EndFunction 
\State $\vv_0 = \vx_0$, $\bar \vx_0 = \vx_0$, $A_0 = 0$ 
\State $\vy_0 = \Call{Cubic-Proximal}{f,\vx_0,\gamma,S,K}$
\State Set $\lambda_0' = \lambda_0= 1/(\gamma \Vert \vy_0 - \vx_0 \Vert) $. 
\For{$t = 0,\cdots, T-1$}   \Comment{Main procedure: optimal Monteiro-Svaiter acceleration}
\State \quad $a_{t+1}' =   \lambda_{t}'(1 + \sqrt{1 + 4  A_t/ \lambda_{t}'}) / 2$, ~ $A_{t+1}' = A_t + a_{t+1}'$,~ $\bar \vx_t = (A_t \vx_t + a_{t+1}' \vv_t)/{A_{t+1}'} $ \label{line:momentum} 
\State \quad \textbf{if} $t > 0 $ \textbf{then} 
\State \quad \quad $\vy_t = \Call{Cubic-Proximal}{f,\bar{\vx}_t,\gamma,S,K}$
\State \quad \quad Set $\lambda_t = 1/(\gamma \Vert \vy_t - \bar \vx_t \Vert)$. 
\State \quad \textbf{end if} 
\State \quad  \textbf{if} $\lambda_{t} \ge  \lambda_{t}'$ \textbf{then} \label{line:begin-seacg-lambda} 
\State \quad \quad $a_{t+1} = a_{t+1}' $, $A_{t+1} = A_{t+1}'$, ~ $\vx_{t+1} = \vy_{t} $, ~ $\lambda_{t+1}' = 2 \lambda_{t}' $
\State \quad \textbf{else} 
\State \quad \quad $\beta_{t} = \lambda_{t}/\lambda_{t}'$,~ $a_{t+1} = \beta_{t} a_{t+1}'$, ~ $A_{t+1} = A_t + a_{t+1}$ 
\State \quad \quad $ \vx_{t+1} = ((1- \beta_{t}) A_t \vx_t + \beta_{t} A_{t+1}' \vy_{t}) / A_{t+1}$,~ $\lambda_{t+1}' = \lambda_{t}'/2$ 
\State \quad \textbf{end if} \label{line:end-search-lambda} 
\State \quad $\vv_{t+1} = \vv_t - a_{t+1} \nabla f(\vy_{t})$ \label{line:extra-ALEN} 
\EndFor \\
\Return $\vx_T$ 
\end{algorithmic}
\end{algorithm*}

\subsection{Reduction to Minimizing a Cubic Proximal Subproblem}
\label{sec:Reduction to a Cubic Proximal Oracle}

The rest of this section focuses on designing an efficient variance-reduced second-order method by leveraging Lemma \ref{lem:bregman-ieq}. We first simplify the design of a completely new procedure by reducing it to minimizing a simpler cubic proximal subproblem. The main reduction we apply is recently proposed by \citet{zhang2026near} for proving near-optimal oracle complexity using lazy second-order oracles, which can be equivalently be interpreted as the complexity using full gradient and stochastic Hessians. Due to the similarity of our setting to theirs, we can reuse several of their techniques with least modifications.

Following \citep{chen2026faster,zhang2026near}, we can apply the optimal Monteiro-Svaiter acceleration \citep{monteiro2013accelerated,carmon2022optimal,kovalev2022first} to design a fast algorithm by catalyzing \citep{lin2018catalyst,frostig2015regularizing} a stochastic sub-solver that solves a series of cubic proximal function
\begin{equation}\label{eq:upper-cubic-objective}
f_{\vx,\gamma}(\vy) \triangleq f(\vy)+\frac{\gamma}{3}\|\vy-\vx\|^3,
\end{equation}
where $\vx \in \sR^d$ is the proximal center and $\gamma>0$ is the regularization.  We restate the full procedure of \citet{zhang2026near} in Algorithm \ref{alg:A-LEN}, where the main procedure is exactly the optimal Monteiro-Svaiter (MS) acceleration by \citet{carmon2022optimal} that works with a cubic proximal oracle that for any $\vx \in \sR^d$ returns $\vy \in \sR^d$ satisfying the MS condition \citep{monteiro2012iteration,monteiro2013accelerated,carmon2022recapp} $\|\nabla f_{\vx,\gamma}(\vy)\|
\le (\gamma/32) \|\vy-\vx\|^2$, and the cubic proximal subroutine follows the design in \citet{zhang2026near} that reduces the suboptimal gap of $f_{\vx,\gamma}$ by a constant contraction factor using a subroutine \texttt{TLN} that will be specified in Algorithm \ref{alg:VR-LEN-Gap}.
Moreover, to relax the MS condition to the following gap reduction condition for the \texttt{TLN} subprocedure, we also apply a restart scheme in the middle and take several cubic regularized Newton(CRN) steps \citep{nesterov2006cubic,nesterov2018lectures} at the beginning and the end:
\begin{equation}\label{eq:upper-cubic-step}
\mathcal{C}_H(\vz)=\vz+\mathop{\rm argmin}_{\vs \in \mathbb{R}^d}
\langle\nabla f_{\vx,\gamma}(\vz),\vs\rangle
+\frac12\vs^{\top}\nabla^2f_{\vx,\gamma}(\vz)\vs
+\frac H3\|\vs\|^3,
\end{equation}
where $H:= L+2\gamma$ is the Hessian Lipschitz constant of $f_{\vx,\gamma}$ according to \citep[Lemma 4.3]{zhang2026near}.

The relaxed gap-reduction condition \citep{chen2026faster,zhang2026near} is now formally stated as follows. Although the previous algorithms using the same conditions have deterministic sub-solvers \citep{chen2026faster,zhang2026near}, we can use the same condition in our randomized sub-solver by repeatedly running a high-probability sub-solver until the condition is satisfied.

\begin{cond} \label{cond:gapreduce-zhang}
Under Assumptions \ref{asm:main}, for any $\Delta_s>0$ and $\vy_s$ satisfying
$f_{\vx,\gamma}(\vy_s)-f_{\vx,\gamma}^*\le\Delta_s$, at least one trial of \texttt{TLN} subroutine in Line \ref{line:repeat-begin} to Line \ref{line:repeat-end} in Algorithm \ref{alg:A-LEN} satisfies 
\begin{align} \label{eq:gap-reduce-34}
f_{\vx,\gamma}(\vy_{s+1})-f_{\vx,\gamma}^*
\le\Delta_{s+1} \triangleq \frac34\Delta_s.    
\end{align}
\end{cond}

We call the final algorithm \texttt{A-TLN}, and the components that are exactly consistent with existing ones \citep{zhang2026near} are presented in Algorithm \ref{alg:A-LEN}. If the above condition is satisfied, then this procedure can achieve the following complexity according to existing analysis.

\begin{lem}[{\citet[Lemma 4.1-4.4]{zhang2026near}}]\label{lem:reduce-gap} Under Assumptions~\ref{asm:main}, let
$\gamma>0$  and $\kappa:=H/\gamma$ be the second-order condition number of $f_{\vx,\gamma}$. (1) If Condition~\ref{cond:gapreduce-zhang} is satisfied for a trial, then this trial breaks at Line~\ref{line:repeat-end}. (2) If Condition~\ref{cond:gapreduce-zhang} is always satisfied, Algorithm~\ref{alg:A-LEN} with the parameters
\begin{align} \label{eq:para-setting-T-S}
T = \mathcal{O}\!\left(
\left(\frac{\gamma R^3}{\epsilon}\right)^{2/7}\right), \quad  S = \gO \left( \log \kappa  \right) 
\end{align}
returns the output $\vx_T \in \sR^d$ that satisfies $f(\vx_T) - f^* \le \eps$.
\end{lem}

Using Lemma \ref{lem:reduce-gap}, we only need to focus on the design of a subroutine that can solve the cubic proximal subproblem efficiently and fulfill Condition \ref{cond:gapreduce-zhang}. However, the subproblem does not satisfy strong convexity, which typically facilitates the design of variance-reduced algorithms such as SVRG~\citep{johnson2013accelerating}. To address this, we introduce a quadratic regularization term to make the problem strongly convex. Meanwhile, to ensure equivalence with the original problem, we further restrict the subproblem within a ball.

Formally, for fixed $\vx,\gamma$ and $s=0$ without loss of generality, given an initialization $\vy_0 \in \sR^d$ satisfying $ f_{\vx,\gamma}(\vy_0) - f^*_{\vx,\gamma} \le \Delta$, we let $\rho = ( 6 \Delta / \gamma)^{1/3} $, $\mu = \Delta / (16 \rho^2)$ and define the following ball-constrained regularized surrogate subproblem $g(\vy) = f_{\vx,\gamma}(\vy) +  \mu \|\vy - \vy_0 \|^2/2$ by
\begin{align} \label{eq:surrogate}
\begin{split}
&\min_{\vy \in \gB} \left\{ g(\vy) := \frac{1}{n} \sum_{i=1}^n g_i(\vy)  \right\}, \quad {\rm where} \quad \gB := \{ \vy \in \sR^d: \| \vy - \vy_0 \| \le \rho \},  \\
&{\rm and} \quad g_i(\vy) := f_i(\vy) + \frac{\gamma}{3} \| \vy - \vx \|^3 + \frac{\mu}{2} \| \vy - \vy_0\|^2. 
\end{split}
\end{align}
The setting of the ball radius $\rho$ ensures that $\gB$ contains the minimizer to $f_{\vx,\gamma}$, and the setting of regularization parameter $\mu$
ensures that $g$ is an approximation of the cubic proximal subproblem~$f_{\vx,\gamma}$ in the following sense:
\begin{lem} \label{lem:surrogate}
Under Assumption \ref{asm:main}, the surrogate function $g: \gB \rightarrow \sR$ defined in \eqref{eq:surrogate} satisfies
\[
f_{\vx,\gamma}(\vy)-f_{\vx,\gamma}^{*}\leq g(\vy)-g^*+\frac\Delta{32},
 \qquad \vy\in \gB.
\]
\end{lem}

Consequently, to find a point $\vy$ that satisfies Condition \ref{cond:gapreduce-zhang} with $f_{\vx,\gamma}(\vy) - f_{\vx,\gamma}^* \leq \Delta/2$, it suffices to find $\vy \in \gB$ with $g(\vy)-g^*\leq 15\Delta/32$.  In the upcoming subsection, we propose a
novel  sub-solver in Algorithm \ref{alg:VR-LEN-Gap} that can achieve this goal by leveraging Lemma \ref{lem:bregman-ieq}.

\subsection{A Taylor-corrected Lazy Newton Sub-solver}

\begin{algorithm}[t]
\caption{\texttt{TLN}$
(f_{\vx,\gamma}, \vy_0, \Delta, H , \gamma, K, \delta)$}
\label{alg:VR-LEN-Gap}
\begin{algorithmic}[1]
\State Set $\rho,\gB,\mu,g_i$ as in~\eqref{eq:surrogate}.
\State  Set $\eta=(32H\rho)^{-1}$ and $J=\lceil\log_2(2n/\delta)\rceil$.
\For{$j=0,\ldots,J-1$}
\State $\vz_0=\vw_0=\vy_j$
\For{$k=0,\cdots,K-1$}
 \State $
      \vv_k=\argmin_{\vv \in \gB}\left\{
       \frac1 {2\eta} \|\vv-\vz_k\|^2
       +\langle\nabla g(\vw_k),\vv-\vw_k\rangle
       +(\vv-\vw_k)^\top\nabla^2 g(\vw_k)(\vv-\vw_k)
      \right\}$
\State Sample $i_k\sim \text{Unif}([n])$ and set $\vr_k=\nabla g_{i_k}(\vv_k)-\nabla g_{i_k}(\vw_k)
       -\nabla^2 g_{i_k}(\vw_k)(\vv_k-\vw_k)$.
\State $\vz_{k+1}=\vv_k-2\eta\vr_k$
\If{$(k+1)\bmod n=0$}
\State Sample $\vw_{k+1}$ uniformly from
            $\{\vv_{k-n+1},\ldots,\vv_k\}$
        \Else
            \State $\vw_{k+1}=\vw_k$
        \EndIf
\EndFor
\State Sample $\vy_{j+1}$ uniformly from $\{ \vv_0,\vv_1,\cdots,\vv_{K-1}\}$
\EndFor
\State \textbf{return} $\vy_J$.
\end{algorithmic}
\end{algorithm}

We now propose our Taylor-corrected Lazy Newton (TLN) sub-solver in Algorithm \ref{alg:VR-LEN-Gap} to minimize the surrogate function (\ref{eq:surrogate}). Following the standard paradigm for variance-reduced methods \citep{defazio2014saga,roux2012stochastic,johnson2013accelerating,shalev2016sdca,tsedenov2026optimal}, the algorithm maintains a snapshot $\vw_k$ that is updated every~$n$ iterations. The full gradient and Hessian oracles are only evaluated at the snapshot, and the other iterations simply take a trust-region Newton step by reusing the information of $\nabla g(\vw_k)$ and $\nabla^2 g(\vw_k)$ from the snapshot:
\begin{align} \label{eq:trust-region-lazy-Newton-step}
    \vv_k=\argmin_{\vv \in \gB}\left\{
       \frac1 {2\eta} \|\vv-\vz_k\|^2
       +\langle\nabla g(\vw_k),\vv-\vw_k\rangle
       +(\vv-\vw_k)^\top\nabla^2 g(\vw_k)(\vv-\vw_k)
      \right\},
\end{align}
where the trust region corresponds to the ball defined in \eqref{eq:surrogate}. Such an update formula is very similar to the lazy Hessian updates \citep{doikov2023second,zhang2026near}, but to further save computation, even the gradient directly reuses the one from the snapshot. However, the lazy update in \eqref{eq:trust-region-lazy-Newton-step} may not converge due to the inexactness in both the gradients and Hessians. To guide the sub-solver towards the minimizer of $g$, our TLN sub-solver takes an additional correction step $\vz_{k+1} = \vv_k - 2 \eta \vr_k$, where $\vr_k$ is a stochastic Taylor remainder estimator:
\begin{align} \label{eq:dfn-rk}
\vr_k=\nabla g_{i_k}(\vv_k)-\nabla g_{i_k}(\vw_k)
       -\nabla^2 g_{i_k}(\vw_k)(\vv_k-\vw_k).
\end{align}
Next, by using this second-order self-bounding inequality established in Lemma \ref{lem:bregman-ieq}, the correction step cancels the error from delayed gradients in \eqref{eq:trust-region-lazy-Newton-step} and eventually forms a convergent sub-solver. We emphasize that this crucial error cancellation relies the individual smoothness and convexity, and here is how our main result improves over \citet{tsedenov2026optimal}. To give a formal convergence analysis, we choose the stepsize as  
$ \eta=(32H\rho)^{-1}$ and define the Lyapunov function by:
\begin{align} \label{eq:lyp-func}
  \Phi(\vz,\vw)=\frac{1}{2}\|\vz-\tilde \vy \|^2+\eta n D_g(\tilde \vy,\vw), \quad {\rm where} \quad \tilde \vy =\arg \min_{\vy \in \gB} g(\vy).
\end{align}
We can show that $\Phi(\vz_k,\vw_k)$ decreases for each step. Then, after telescoping over the iterations in the epochs, we can obtain the following result.
\begin{lem}\label{lem:one-epoch}
Under Assumption \ref{asm:main}, for any $s \in \{ 0, 1, \cdots,  K/n -1\}$, the two consecutive snapshots of the inner loop of Algorithm \ref{alg:VR-LEN-Gap} satisfies the following descent property:
\[
\E \left[\Phi(\vz_{sn},\vw_{sn})
 -\Phi(\vz_{(s+1)n},\vw_{(s+1)n}) \right] \geq\eta\sum_{k=sn}^{(s+1)n-1} \E \bigl[ g(\vv_k)- g(\tilde \vy)+D_g(\tilde \vy,\vv_k)\bigr]. 
\]
\end{lem}

Next, by the strong convexity of $g$, by restarting the above updates every $K$ iterations at $\vy_{j+1}$ sampled uniformly from $\{ \vv_0,\cdots,\vv_{K-1}\}$, the potential $\Phi(\vy_j,\vy_j)$ contracts by a factor each time the algorithm restarts.

\begin{cor} \label{cor:one-epoch}
Under Assumption \ref{asm:main}, for any $j \in \{0,\cdots,J-1\}$, each time the outer of Algorithm \ref{alg:VR-LEN-Gap} restarts with a new $\vy_{j+1}$, the Lyapunov function contracts by a factor as follows:
\[
\E \Phi(\vy_{j+1}, \vy_{j+1}) \leq \frac{n+1/(\eta\mu)}K \Phi(\vy_j,\vy_j).
\]
\end{cor}

Then, by comparing $\Phi$ with both the left- and right-hand sides in Condition \ref{cond:gapreduce-zhang}, we can obtain the complexity of the TLN sub-solver to let Line \ref{line:repeat-end} in the Cubic-Proximal function of Algorithm~\ref{alg:A-LEN} break after appropriately choosing the parameters.

\begin{lem}
\label{lem:ensure-gap-reduce}
Let 
$\rho = ( 6 \Delta / \gamma)^{1/3} $ and $\mu = \Delta / (16 \rho^2)$ in the surrogate function of \eqref{eq:surrogate}.
For an initialization $\vy_0 \in \sR^d$ satisfying $ f_{\vx,\gamma}(\vy_0) - f^*_{\vx,\gamma} \le \Delta$, choosing 
\texttt{TLN} with parameters $K = \gO(n+\kappa)$  as the sub-solver in the Cubic-Proximal function of Algorithm~\ref{alg:A-LEN} returns an output that satisfies the condition in Line \ref{line:repeat-end}, \textit{i.e.}, $\|\nabla f_{\vx,\gamma}(\hat{\vy})\|\leq (\gamma/32)\|\hat{\vy}-\vx\|^2$ in the expected number of SSO calls bounded by
\begin{align}
      \label{eq:ISO calls for algorithm 2}\gO((n+\kappa)\log \kappa \log (n\log \kappa)),
\end{align}  
where $\kappa: = H / \gamma$ is the second-order condition number of the cubic proximal subproblem $f_{\vx,\gamma}(\vy)$.


\end{lem}

\subsection{Main Theorem}

Now, we can prove our main theorem for the complexity of the proposed A-TLN method in Algorithm~\ref{alg:A-LEN} by combining the outer loop guarantee in Lemma \ref{lem:reduce-gap} and the inner loop complexity in Lemma \ref{lem:ensure-gap-reduce}. The optimal choice of the cubic regularization parameter is $\gamma = L/n$, which makes $\kappa = n+2$ and each cubic proximal call in Algorithm \ref{alg:A-LEN} costs $\tilde \gO(n)$ in expectation. Finally, we obtain the complexity of A-TLN as follows. 

\begin{thm}\label{thm:upper-bound}
Under Assumption~\ref{asm:main}, initialize at $\vx_0 = \vzero$, 
run Algorithm~\ref{alg:A-LEN} 
and Algorithm~\ref{alg:VR-LEN-Gap} as the subroutine with the parameters 
\begin{align} \label{eq:para-main}
 \gamma=\frac Ln,\quad T = \mathcal{O}\!\left(
\left(\frac{\gamma R^3}{\epsilon}\right)^{2/7}\right), \quad S = \gO(\log \kappa), \quad K = \gO(n + \kappa).
\end{align}
The method terminates with probability one, returns an $\epsilon$-solution $\vx_T \in \sR^d$ that satisfies $f(\vx_T) - f^* \le \eps$, and the expected number of SSO calls is bounded by
\begin{equation}
\gO\left( \left(n+n^{5/7}\left(\frac{LR^3}{\epsilon}\right)^{2/7} \right) \log^2 n\right).
 \label{eq:overall-complexity}
\end{equation}
\end{thm}

By comparing the result with the lower bound in Theorem \ref{thm:lower}, we can conclude that the complexity of our proposed A-TLN method is optimal up to only a multiplicative logarithmic factor of $\log^2 n$. We leave the development of optimal methods to future work.

\section{Conclusion}

In this paper, we resolve the optimal stochastic second-order oracle complexity for convex finite-sum optimization under individual smoothness and convexity. We prove a zero-respecting lower bound of $\Omega(n + n^{5/7} \eps^{-2/7})$ and a nearly matching upper bound of $\tilde \gO(n + n^{5/7} \eps^{-2/7})$ using a novel A-TLN method. Our upper bound strictly improves the $\tilde \gO(n+n^{6/7} \eps^{-2/7})$ upper bound achieved by the recently proposed optimal stochastic Newton method in \citep{tsedenov2026optimal} under averaged smoothness and convexity. Our results thus show a clear separation for these two function classes. Moreover, our new upper bound relies on a new second-order self-bounding inequality, which is of independent interest.

\bibliography{sample}
\bibliographystyle{plainnat}

\appendix
\section{Missing Proofs of the Lower Bound}

\subsection{Proof of Lemma \ref{lem:chain}}

\begin{proof}
For $k=0,\ldots,2K$, define the truncated chain
\[
  f^{(k)}(\vx)=\frac{\lambda}{n}\left(|a-x_{[1]}|^3+\sum_{j=2}^{k}|x_{[j-1]}-x_{[j]}|^3+|x_{[k]}|^3\right),
\]
with $f^{(0)}(\vx)=\lambda a^3/n$, and let $f_i^{(k)}$ denote the corresponding component truncations. Define the number of revealed coordinates $N_t$ recursively as $N_0 = 0$ and, for $t \ge 1$:
\[
N_t =
\begin{cases}
    N_{t-1}, & N_{t-1} = 2K; \\
    N_{t-1} + 1, & N_{t-1} \neq 2K \text{ and } (t-1) \bmod n = 0; \\
    N_{t-1} + \mathbf{1}\{I_t=c(N_{t-1}+1)\}, & N_{t-1} \neq 2K \text{ and } (t-1) \bmod n \neq 0.
\end{cases}
\]

By Definition \ref{dfn:zero-respecting-alg}, any zero-respecting algorithm satisfies $\operatorname{supp}(\vx_{t-1}) \subseteq \bigcup_{s=1}^{t-1} \operatorname{supp}(\sO_s)$. Because only the first $N_t$ coordinates have been activated and revealed by the oracle up to time $t$, all the $j$-th coordinates of $\vx_{t-1}$ are strictly zero for $j > N_{t-1}$. Thus, the second-order oracle response of $f_i$ at $\vx_{t-1}$ coincides with the one given by the truncated function $f_i^{(N_t)}$, which means that the $j$-th coordinate of $\vx_t$ is zero when $j>N_t$.


Finally, each snapshot increases $N_t$ by at most 1, while non-snapshot steps increase $N_t$ by 1 with probability at most $1/n$. Taking expectations yields $\E N_T \le 3T/n$. By Markov's inequality with $K = \lceil 8T/n \rceil$,
\[
    \Prob(N_T \ge 2K) \le \frac{\E N_T}{2K} \le \frac{3}{16}.
\]
On the complementary event ($N_T < 2K$), the last coordinate of $\vx_T$ remains zero, completing the proof.
\end{proof}



\subsection{Proof of Lemma \ref{lem:gap}}

\begin{proof}
For any vector $\vx \in \sR^{2K}$ that satisfies $\vx_{[2K]} = 0$, we have $f(\vx)=f^{(2K-1)}(\vx)$, which means
\[
f(\vx) = \frac{\lambda}{n}\left(|a-x_{[1]}|^3+\sum_{j=2}^{2K-1}|x_{[j-1]}-x_{[j]}|^3+|x_{[2K-1]}|^3\right) \ge \frac{\lambda a^3}{4 n K^2},
\]
where the last step uses Jensen's inequality and the observation that
the $2K$ terms $a-x_{[1]},x_{[1]}-x_{[2]},\ldots,x_{[2K-1]}$ sum to $a$. Finally, from the form of function $f$ in \eqref{eq:instance}, we have $f^\ast = 0$, obtained at
$x^\star=a\sum_{j=1}^{2K}\ve_j$, where $\ve_j$ denote the $j$-th standard canonical basis vector in $\sR^{2K}$.
\end{proof}

\subsection{Proof of Theorem \ref{thm:lower}}

\begin{proof}
Let $\mB \in \sR^{2K \times 2K}$ and $\mD(\vx)\in \sR^{2K \times 2K}$ be the matrix defined by 
\[
\mB = \begin{pmatrix}
-1 & 0 & 0  & \cdots & 0 & 0 \\
1 & -1 & 0  & \cdots & 0 & 0 \\
0 & 1 & -1  & \cdots & 0 & 0  \\
\vdots & \vdots & \vdots & \ddots & \vdots & \vdots \\
0 & 0 & 0 & \cdots & 1 & -1 
\end{pmatrix}, \quad
\mD(\vx) = \begin{pmatrix}
|a - x_{{[1]}}| & 0 & \cdots & 0 \\
0 & |x_{[1]} - x_{[2]}|  & \cdots & 0 \\
\vdots  & \vdots & \ddots & \vdots \\
0 & 0  & \cdots & |x_{[2K-1]} - x_{[2K]}|
\end{pmatrix}. \]
Then the Hessian of the function $f$ in \eqref{eq:instance} satisfies $\nabla^2 f(\vx) = \frac{6 \lambda}{n} \mB^\top \mD(\vx) \mB$. Moreover, denote $J_i=\{j\in[2K]:c(j)=i\}$ as the index set assigned to component $f_i$ and 
let $\mB_i$ be the submatrix consist of the rows of $\mB$ indexed by $J_i$ and $\mD_i(\vx)$ be its principal submatrix indexed by $J_i$, respectively. The Hessian of each component function $f_i$ satisfies 
$\nabla^2 f_i(\vx) = 6\lambda \mB_i^\top \mD_i(\vx) \mB_i$. Since $\norm{\mB_i}\le\norm \mB\le2$ and every row of~$\mB_i$
has norm at most $\sqrt2$, setting $\lambda=L/48$ gives
\begin{align*}
    \norm{\nabla^2 f_i(\vx)-\nabla^2 f_i(\vy)}
    &\le6\lambda\norm{\mB_i}^2
        \norm{\mD_i(\vx)-\mD_i(\vy)} \le24\sqrt2\lambda\norm{\vx-\vy}
    \le L\norm{\vx-\vy},
\end{align*}
which means that each component function $f_i$ has $L$-Lipschitz continuous Hessians.
Since $f$ achieves its minimum at $\vx^\ast=a\sum_{j=1}^{2K}\ve_j$, we have $f^\ast=0$ and
$\norm{\vx^\ast}=R$ when $a=R/\sqrt{2K}$. Therefore, the hard instance satisfies Assumption~\ref{asm:main} under our parameter setting. By Lemma~\ref{lem:chain}, we have $(\vx_{T})_{[2K]}$ holds with
probability at least $13/16$. On this event, Lemma~\ref{lem:gap} yields the claimed lower bound:
\[
    f(\vx_T)-f^\ast
    \ge\frac{\lambda a^3}{4nK^2}
    =\frac{LR^3}{48n(2K)^{7/2}}.
\]
\end{proof}

\section{Missing Proofs of the Upper Bound}

Before formally proving the lemmas in the main text, we first recall several useful lemmas. The first one  summarizes several important properties of the cubic proximal functions from \citep{chen2026faster,zhang2026near,nesterov2018lectures}.

\begin{lem}[{\citet[Lemma 4.2]{zhang2026near}}] \label{lem:sub-prob-grad-dom}
Under Assumptions \ref{asm:main}, the function $f_{\vx,\gamma}(\vy)$ defined in Eq.~(\ref{eq:upper-cubic-objective}) satisfies:
\begin{enumerate}
    \item $f_{\vx,\gamma}(\vy)$ has $(L+2 \gamma)$-Lipschitz continuous Hessians.
    \item For any $\vy \in \sR^d$ and $\vy_{\vx,\gamma}^* = \arg \min_{\vy \in \sR^d} f_{\vx,\gamma}(\vy)$, we have
    \begin{align}
        \Vert \nabla f_{\vx, \gamma}(\vy) \Vert^{3/2} \ge \frac{3}{2} \sqrt{\frac{\gamma}{2}} (f_{\vx, \gamma}(\vy) - f_{\vx,\gamma}^*)
        \geq \frac{1}{2}\left( \frac{\gamma}{2} \Vert \vy - \vy_{\vx,\gamma}^* \Vert^2 \right)^{3/2}.\label{eq:cubic-grad-uniform}
    \end{align}
\end{enumerate}
\end{lem}

Next, we present a lemma that estimates the difference of two Bregman divergences after exchanging the order of two variables.

\begin{lem} \label{lem:Dg-Lipschitz}
If $g: \mathbb{R}^d \rightarrow \sR$ is convex and has $H$-Lipschitz continuous Hessians, then, for all
$\vv,
\vw\in\mathbb{R}^d$,
\[
 |D_g(\vv,\vw)-D_g(\vw,\vv)|\le\frac H6\| \vv - \vw \|^3.
\]
\end{lem}

\begin{proof}
For convenience, we set $\vs = \vv - \vw$ and define the quadratic form
$q(t)=\vs^\top\nabla^2g(\vw+t \vs)\vs$. The definition of Bregman divergence $D_g$ and Taylor's integral formula indicate that
\[
 D_g(\vv,\vw)=\int_0^1(1-t)q(t)\,{\rm d}t,\qquad
 D_g(\vw,\vv)=\int_0^1t\,q(t)\,{\rm d}t.
\]
Since $\int_0^1(1-2t)\,{\rm d} t=0$ and
$|q(t)-q(1/2)|\le H\|\vs\|^3|t-1/2|$ under the Hessian Lipschitz condition, we have
\begin{align*}
 |D_g(\vv,\vw)-D_g(\vw,\vv)|
 &=\left|\int_0^1(1-2t)\bigl(q(t)-q(1/2)\bigr)\,{\rm d}t\right|\le H\|\vs\|^3\int_0^1|1-2t|\,|t-1/2|\,{\rm d} t =\frac H6\|\vs\|^3.
\end{align*}

\end{proof}

\subsection{Proof of Lemma \ref{lem:bregman-ieq}}


\begin{proof}
Let $\vs=\vv-\vw$, and for
$\tau\in(0,1)$, we introduce the following Jensen gap as an auxiliary objective:
\[
 J_\tau(\vw):=\tau f(\vw+\vs)+(1-\tau)f(\vw)-f(\vw+\tau \vs).
\]
Convexity of the function $f$ gives $J_\tau(\vw)\ge0$ for every $\vw \in\mathbb R^d$. Moreover, its Hessian is 
\begin{align*}
 \nabla^2J_\tau(\vw)
 =\tau\bigl[\nabla^2 f(\vw+\vs)-\nabla^2 f(\vw+\tau \vs)\bigr]+(1-\tau)\bigl[\nabla^2 f(\vw)-\nabla^2 f(\vw+\tau \vs)\bigr].
\end{align*}
Then, since $f$ has $L$-Lipschitz continuous Hessians, we have
\[
 \|\nabla^2J_\tau(\vw)\|
 \le 2\tau(1-\tau)L\|\vs\|=:\ell_\tau,
\]
which means that $J_\tau$ has $\ell_\tau$-Lipschitz gradients. Applying the descent inequality and its nonnegativity yields 
\[
 0\le J_\tau\!\left(\vw- \frac{\nabla J_\tau(\vw)}{\ell_\tau}\right)
 \le J_\tau(\vw)-\frac{\|\nabla J_\tau(\vw)\|^2}{2\ell_\tau}.
\]
Dividing the above inequality by $\tau^2$ and rearranging gives
\[
 \left\|\frac{\nabla J_\tau(\vw)}{\tau}\right\|^2
 \le 2\ell_\tau  \frac{J_\tau(\vw)}{\tau^2} 
 = 4(1-\tau)L\|\vs\|\frac{J_\tau(\vw)}{\tau}.
\]
Passing the limit to $t\downarrow0$, we know that $\nabla J_\tau (\vw) / \tau$ and $J_{\tau}(\vw)/ \tau$ in the above converges as follows
\begin{align*}
 \lim_{\tau\downarrow0}\frac{\nabla J_\tau(\vw)}{\tau}
 &=\nabla f(\vv)-\nabla f(\vw)-\nabla^2 f(\vw) \vs,\\
 \lim_{\tau\downarrow0}\frac{J_\tau(\vw)}{\tau}
 &=f(\vv)-f(\vw)-\langle\nabla f(\vw),\vs\rangle=D_f(\vv,\vw),
\end{align*}
which then proves the desired inequality.
\end{proof}

\subsection{Proof of Lemma \ref{lem:surrogate}}

\begin{proof}
By the definition of the surrogate function $g_i$ in \eqref{eq:surrogate}, it is easy to see that each $g_i$ is $\mu$-strongly convex and has $H$-Lipschitz continuous Hessians for $H = L + 2\gamma$. Lemma \ref{lem:sub-prob-grad-dom} gives $\|y^*_{\vx,\gamma}-\vy_0\|^3\leq 6 \Delta/\gamma$ for $\vy_{\vx,\gamma}^* = \arg \min_{\vy \in \sR^d} f_{\vx,\gamma}(\vy)$, hence, $\vy^*_{\vx,\gamma}\in \gB := \{ \vy \in \sR^d: \| \vy - \vy_0 \| \le (6 \Delta/\gamma)^{1/3}\}$. Consequently, since $g(\vy) = f_{\vx,\gamma}(\vy) + \mu \| \vy - \vy_0  \|^2/2$, we have
\[
g^* \leq g(\vy_{\vx,\gamma}^{*})=f^*_{\vx,\gamma}+\frac{\mu}{2}\|\vy^*_{\vx,\gamma}-\vy_0 \|^2.
\]
Then, we can further obtain that
\[
f_{\vx,\gamma}(\vy) - f_{\vx,\gamma}^* \le g(\vy) - g^* + \frac{\Delta}{32}, \quad \forall \vy \in \gB.
\]
\end{proof}

\subsection{Proof of Lemma \ref{lem:one-epoch}}
\begin{proof}
Fix an epoch $s$ and let $k\in\{sn,\ldots,(s+1)n-1\}$ be an arbitrary iteration in the epoch. Note that the snapshot $\vw_s$ remains fixed throughout the whole epoch. The optimality condition of 
$\vv_k$ in the Newton step and $\tilde \vy \in \gB$ give
\[
 \langle\vz_k-\vv_k,\vv_k-\tilde \vy \rangle
 \geq\eta\langle\nabla g(\vw_s)
        +2\nabla^2 g(\vw_s)(\vv_k-\vw_s),\vv_k-\tilde \vy \rangle.
\]
Since $i_k$ is sampled uniformly, we have
 \[
 \E_k\vr_k=\nabla g(\vv_k)-\nabla g(\vw_s)
            -\nabla^2 g(\vw_s)(\vv_k-\vw_s).
\]
Now, expanding $\|\vz_{k+1}-\tilde \vy\|^2$ and using these two relations yields
\[
 \begin{aligned}
 &\frac12\|\vz_k-\tilde \vy \|^2
 -\frac12\E_k\|\vz_{k+1}-\tilde \vy \|^2
 \\
 =& \langle \vz_k - \vv_k, \vv_k  -\tilde \vy \rangle + \frac{1}{2} \| \vz_k  - \vv_k \|^2 + 2 \eta \langle \E_k r_k, \vv_k - \tilde \vy \rangle - 2 \eta^2 \E_k \| \vr_k\|^2 \\
 \geq &\eta\langle2\nabla g(\vv_k)-\nabla g(\vw_s),\vv_k-\tilde \vy \rangle +\frac12\|\vz_k-\vv_k\|^2-2\eta^2\E_k\|\vr_k\|^2.
 \end{aligned}
\]
By Lemma \ref{lem:sub-prob-grad-dom}, each component function $g_i$ has $H := (L+2\gamma)$-Lipschitz continuous Hessians.
Averaging Lemma~\ref{lem:bregman-ieq} over all the $n$ component functions $g_i$ 
and using $\|\vv_k-\vw_s\|\leq2\rho$ gives
$\E_k\|\vr_k\|^2\leq8H\rho D_g(\vv_k,\vw_s)$. The definition of the Bregman divergence gives the following identity:
\[
 \begin{aligned}
 \langle2\nabla g(\vv_k)-\nabla g(\vw_s),\vv_k-\tilde \vy \rangle
 &=g(\vv_k)-g(\tilde \vy)+2D_g(\tilde \vy,\vv_k) - D_g(\tilde \vy,\vw_s)+D_g(\vv_k,\vw_s).
 \end{aligned}
\]
Then, the choice of stepsize of $\eta=(32H\rho)^{-1}$ therefore imply
\[
 \begin{aligned}
 \frac12\|\vz_k-\tilde \vy \|^2
 -\frac12\E_k\|\vz_{k+1}-\tilde \vy \|^2
 &\geq\eta\bigl[g(\vv_k)-g(\tilde \vy)+2D_g(\tilde \vy,\vv_k)
                  -D_g(\tilde \vy,\vw_s)\bigr]\\
 &\quad+\frac12\|\vz_k-\vv_k\|^2
            +\frac\eta2D_g(\vv_k,\vw_s).
 \end{aligned}
\]
Taking full expectation over the whole epoch, telescoping over $k = sn,\cdots, (s+1)n-1$, and discarding the last two nonnegative terms in the above gives
\[
 \begin{aligned}
 \frac12\|\vz_{sn}-\tilde \vy \|^2
 -\frac12\E \|\vz_{(s+1)n}-\tilde \vy \|^2
 &\geq\eta\sum_{k=sn}^{(s+1)n-1}
    \E \bigl[g(\vv_k)-g(\tilde \vy)+2D_g(\tilde \vy,\vv_k)\bigr]- \eta nD_g(\tilde \vy,\vw_s).
 \end{aligned}
\]
At the end of the epoch,  $\vw_{(s+1)n}$ is sampled uniformly from
 $\{\vv_{sn},\ldots,\vv_{(s+1)n-1}\}$, indicating that \( 
 \E D_g(\tilde \vy, \vw_{(s+1)n}) = 1/n \cdot \sum_{k=sn}^{(s+1)n-1} \E D_g(\tilde \vy,\vv_k) \). Substituting this identity and the definition of $\Phi$ proves
 \begin{align} \label{eq:Phi-descent-one-epoch}
\E \left[\Phi(\vz_{sn},\vw_{sn})
 -\Phi(\vz_{(s+1)n},\vw_{(s+1)n}) \right] \geq\eta\sum_{k=sn}^{(s+1)n-1} \E \bigl[ g(\vv_k)- g(\tilde \vy)+D_g(\tilde \vy,\vv_k)\bigr],
\end{align}
 which is exactly as desired by the lemma.
 


\end{proof}

\subsection{Proof of Corollary \ref{cor:one-epoch}}

\begin{proof}
For the $j$th round, taking expectations and summing over \eqref{eq:Phi-descent-one-epoch} over $k=0,\ldots,N/n-1$ gives
\[
 \eta\sum_{k=0}^{K-1}\E\bigl[g(\vv_k)-g(\tilde \vy)+D_g(\tilde \vy,\vv_k)\bigr] \leq\Phi(\vz_0,\vw_0) - \Phi(\vz_K, \vw_K) \le \Phi(\vz_0,\vw_0),
\]
where the last step uses the non-negativity of the Lyapunov function.
Since $\vy_{j+1}$ is sampled uniformly from $\{ \vv_0,\vv_1,\cdots,\vv_{K-1}\}$ and $\vz_0 = \vw_0 = \vy_j$ in the $j$th round, the above inequality implies that
\begin{align} \label{eq:one-round-bregman}
     \E\bigl[g(\vy_{j+1})-g(\tilde \vy)+D_g(\tilde \vy,\vy_{j+1})\bigr] \le \frac{\Phi(\vy_j,\vy_j)}{\eta K}.
\end{align}
Next, we use the strong convexity of $g$ to obtain  $ \mu\| \vy_{j+1} - \tilde \vy \|^2 / 2 \le D_g(\tilde \vy, \vy_{j+1}) $ and substitute this bound into $\Phi(\vy_{j+1},\vy_{j+1})$ to get
\begin{align} \label{eq:Phi-descent-one-restart}
\E \Phi(\vy_{j+1}, \vy_{j+1}) \le \left( \frac{1}{\mu} + n \eta \right) \E D_g(\tilde \vy, \vy_{j+1})  \le \left( \frac{1}{\mu} + n \eta \right) \frac{\Phi(\vy_j,\vy_j)}{\eta K}.    
\end{align}
\end{proof}

\subsection{Proof of Lemma \ref{lem:ensure-gap-reduce}}

\begin{proof}
In light of Lemma \ref{lem:sub-prob-grad-dom}, we define $\kappa:= H / \gamma =  2 + L/ \gamma$ be the second-order condition number of both the functions $f_{\vx,\gamma}$ and $g$.
The
constrained optimality of $\tilde \vy = \arg \min_{\vy \in \gB} g(\vy)$ yields
$\langle\nabla g(\tilde \vy),\vy_0-\tilde \vy \rangle\ge0$. Since $g(\vy_0)=f_{\vx,\gamma}(\vy_0)$ and
$g(\tilde \vy)\ge f_{\vx,\gamma}^\star$, we have
\[
 D_g(\vy_0,\tilde \vy)\le g(\vy_0)-g(\tilde \vy)\le\Delta.
\]
Together with Lemma \ref{lem:Dg-Lipschitz} and $\tilde \vy \in \gB := \{ \vy \in \sR^d: \| \vy - \vy_0 \| \le \rho \}$,
this implies
\begin{equation}
 D_g(\tilde \vy,\vy_0)\le\Delta+\frac H6\rho^3=(1+\kappa)\Delta,
 \label{eq:vrlen-initial-divergence}
\end{equation}
where the last step is due to $\rho^3  = 6 \Delta / \gamma$.
Recalling the definition of $\Phi$ in \eqref{eq:lyp-func}, we know that
\begin{align} \label{eq:vrlen-initial-potential}
 \Phi(\vy_0,\vy_0) = \frac{1}{2} \| \vy_0 - \tilde \vy\|^2 + \eta n D_g(\tilde \vy, \vy_0) \le 
\frac12\rho^2+\eta n(1+\kappa)\Delta\nonumber =\left(\frac12+\frac{n(1+\kappa)}{192\kappa}\right)\rho^2
 \le n\rho^2,
\end{align}
where the second last step uses $ \eta = (32 H \rho)^{-1} $ and $\Delta = \rho^3 \gamma/ 6 $ and last inequality is due to $n\ge1$ and $\kappa\ge2$. Since $1/(\eta\mu)=3072\kappa$, choosing
$K = \gO(n+\kappa)$ can ensure $K \geq 2(n+1/(\eta\mu))$ and Lemma \ref{lem:one-epoch} therefore reduces the expected value of
$\Phi(\vy_j,\vy_j)$ by a factor of two in each restart. After all $J-1$ times of restarts,
\[
\E \Phi(\vy_{J-1}, \vy_{J-1}) \le 2^{-(J-1)} \Phi(\vy_0,\vy_0). 
\]
Applying \eqref{eq:Phi-descent-one-restart} to the last round with $j=J-1$ and  using $1/(\eta K)\leq\mu/2$, we have
\[
\E [g(\vy_J) - g(\tilde \vy) ] \le \frac{2^{-(J-1)}}{\eta K} \Phi(\vy_0,\vy_0) \le 2^{-J} \mu \Phi(\vy_0,\vy_0)  \le 2^{-J} \mu n \rho^2 = \frac{n\Delta}{16 \cdot 2^J} \le \frac{\delta \Delta}{32},
\]
where the second last step uses the parameter setting of $\mu = \Delta / (16 \rho^2)$ and the last inequality holds when $J = \gO(\log (n / \delta))$.
By Lemma \ref{lem:surrogate} and Markov's inequality,
\[
 \Prob\bigl(f_{\vx,\gamma}(\vy_J)-f_{\vx,\gamma}^{*}>\Delta/2\bigr)
 \leq\Prob\bigl(g(\vy_J)-g^* >15\Delta/32\bigr)
 \leq\frac\delta{15}.
\]
Therefore,  the output of the \texttt{TLN} subroutine satisfies Condition \ref{cond:gapreduce-zhang} with probability at least $1- \delta/15$, which is sufficient for Line \ref{line:repeat-end} to break by item (1) in Lemma \ref{lem:reduce-gap}.
Using a union bound for all the $S$ epochs in the restart scheme, each trial from Line \ref{line:repeat-begin} to Line \ref{line:repeat-end} breaks with probability at least $1 - S \delta \ge 3/4$. Then, the probability of requiring more than $k$
repetitions until successfully returns an output satisfying the condition in \ref{line:repeat-end} is at most $4^{-k}$. By the expectation of Geometric distribution, the expected total number of trials is at most $4/3$. Combining the parameter setting of $ S = \log (\kappa) $, $\delta  =  1/ (4 S)$, $J = \gO(\log (n / \delta))$, and $K = \gO( n+ \kappa)$ means that the total SSO calls to return  $\|\nabla f_{\vx,\gamma}(\hat{\vy})\|\leq (\gamma/32)\|\hat{\vy}-\vx\|^2$ in expectation is bounded by
\[
\frac{4}{3} S J ( n+K) = \gO( (n + \kappa) \log \kappa \log (n \log \kappa) ).
\]

\end{proof}


\subsection{Proof of Theorem \ref{thm:upper-bound}}

\begin{proof}
Setting the cubic regularization parameter $\gamma = L/n$, then the parameters in \eqref{eq:para-main} satisfies $T = \gO(1 + n^{-2/7} \epsilon^{-2/7})$, $S = \log n$, $J = \log (n \log n)$, and $K = \gO(n)$. Therefore, by Lemma \ref{lem:reduce-gap} and \ref{lem:ensure-gap-reduce}, the total SSO complexity of Algorithm \ref{alg:A-LEN} is $ \gO(T S J (K+n)) =  \gO( (n + n^{5/7} \epsilon^{-2/7}) \log^2 n )$ in expectation.
\end{proof}

\end{document}